\documentclass[11pt]{amsart}
\usepackage[margin=1in]{geometry}
\usepackage{amsmath}
\usepackage{amsfonts}
\usepackage{adjustbox}
\usepackage{stackengine}
\usepackage{amsthm}
\usepackage{blkarray}
\usepackage{verbatim}
\usepackage{mathrsfs}
\usepackage{amssymb}
\usepackage{mathtools}
\usepackage{xcolor}
\usepackage{algorithm}
\usepackage{algorithmic}
\usepackage{bbold}
\usepackage{cite}
\usepackage{latexsym}
\usepackage{booktabs}
\usepackage{bm}
\usepackage{graphicx}
\usepackage{textcomp}
\usepackage{subfigure}
\usepackage[toc,page]{appendix}
\usepackage{url}
\usepackage{hyperref}
\usepackage{multicol}
\usepackage[thinlines]{easytable}
\usepackage{multirow}
 \usepackage{todonotes}
\usepackage{tikz}
\usepackage{fancyhdr} 
\usepackage{array}
\usepackage{array}
\usepackage{makecell}

\let\oldtextit\textit 
\renewcommand\textit[1]{\oldtextit{\color{RoyalBlue}#1}}
\definecolor{RoyalBlue}{cmyk}{1, 0.50, 0, 0}
\definecolor{Amber}{rgb}{1.0, 0.49, 0.0}
\definecolor{zzwwqq}{rgb}{0.6,0.4,0}

\theoremstyle{definition}
\newtheorem{definition}{Definition}[section]

\newtheorem{corollary}[definition]{Corollary}
\newtheorem{lemma}[definition]{Lemma}

\newtheorem{theorem}[definition]{Theorem}

\newtheorem{exampleth}[definition]{Example}
\newenvironment{example}{\begin{exampleth}}{\hfill $\diamond$\\ \end{exampleth}}

\newtheorem{remark}[definition]{Remark}

\newcommand{\CC}{\mathbb C}
\newcommand{\RR}{\mathbb R}

\newcommand{\cR}{\mathcal R}

\newcommand{\cM}{\mathcal M}

\newcommand{\kk}{\mathbf k}
\newcommand{\cC}{\mathcal C}
\newcommand{\cK}{\mathcal K}

\newcommand{\cS}{\mathcal S}

\newcommand{\Mdnr}{{{\cM}^r_{d,n}}}
\newcommand{\Snr}{{{\cS}^r_n}}

\newcommand{\tropR}{\text{trop}_{\mathcal{R}^+}}
\newcommand{\tropRg}{\text{trop}_{\mathcal{R}}}
\newcommand{\tropC}{\text{trop}_{\mathcal{C}^+}}
\newcommand{\tropCg}{\text{trop}_{\mathcal{C}}}
\newcommand{\trop}{\text{trop}}
\newcommand{\puiseux}{\{\!\{t\}\!\}}
\newcommand{\red}[1]{\textcolor{red}{#1}}
\newcommand{\blue}[1]{\textcolor{blue}{#1}}
\DeclareMathOperator{\val}{val}
\makeatletter
\def\@settitle{\begin{center}%
  \baselineskip13\p@\relax
    \Large
\@title
  \end{center}%
}
\makeatother
\title{Real and Positive Tropicalizations of Symmetric Determinantal Varieties}

\author{Abeer Al Ahmadieh}
\email{aahmadieh3@gatech.edu}

\author{May Cai}
\email{mcai@gatech.edu}

\author{Josephine Yu}
\email{jyu@math.gatech.edu}
\address{School of Mathematics, Georgia Tech, Atlanta GA 30332, US}

\date{}

\begin{document}

\begin{abstract}
    We study real and positive tropicalizations of the varieties of low rank symmetric matrices over real or complex Puiseux series.  We show that real tropicalization coincides with complex tropicalization for rank two and corank one cases.  We also show that the two notions of positive tropicalization introduced by Speyer and Williams coincide for symmetric rank two matrices, but they differ for symmetric corank one matrices.   
\end{abstract}
\maketitle
\section{Introduction}

Let $\cC = \CC\puiseux$ and $\cR=\RR\puiseux$ be the fields of Puiseux series over complex and real numbers  respectively.  The tropical semiring $(\RR, \oplus,\odot)$ is the set of real numbers with the tropical addition and the tropical multiplication given by $x\oplus y = \min(x,y)$ and $x\odot y = x+y$ respectively.
Then the Puiseux series fields admit a valuation to the tropical semiring
\[
\val: \cC \text{ or } \cR \rightarrow \RR \text{ given by } \val(a)= \deg({\rm lt}(a))
\]
where ${\rm lt}(a)$ is the leading term, or the lowest degree term, of a Puiseux series $a$.

For an algebraic variety $X$ in $\cC^n$ or  $(\cC^*)^n$, its (complex) tropicalization is defined to be 
\[
\trop_\cC(X) = \overline{\{\val(x) \mid x \in X \cap (\cC^*)^n \}}\subset \RR^n
\]
where the valuation is taken coordinatewise and the closure is under Euclidean topology.  The Fundamental Theorem and the Structure Theorem hold for these tropical varieties~\cite{MaclaganSturmfels}.  Analogously, we can consider points in $X$  with real coordinates and define its \textit{real tropicalization}:
\[
\trop_\cR(X) = \overline{\{\val(x) \mid x \in X \cap (\cR^*)^n\}}\subset \RR^n.
\]

Let $\cC^+$ and $\cR^+$ be the subsets of $\cC$ and $\cR$ respectively with positive leading terms.  Following \cite{BLS}, we define the \textit{positive part} of the tropical variety to be
\[
\trop_{\cC^+}(X) = \overline{\{\val(x) \mid x \in X \cap (\cC^+)^n\}}\subset \RR^n.
\]
and the \textit{really positive part} to be
\[
\trop_{\cR^+}(X) = \overline{\{\val(x) \mid x \in X \cap (\cR^+)^n\}}\subset \RR^n.
\]
The real tropicalization $\trop_{\cR}(X)$ and the really positive tropicalization  $\trop_{\cR^+}(X)$ arise as real tropicalizations of real semialgebraic sets and satisfy the Fundamental Theorem of Real Tropical Geometry~\cite{JSY}. Their real tropical bases were studied in \cite{tabera2015real}. When $X$ is defined over $\RR$ they coincide with the logarithmic limit sets of (the positive part of) the real variety.  

On the other hand, the positive tropicalization  $\trop_{\cC^+}(X)$ is a subcomplex of the Gr\"obner complex of the ideal defining $X$, corresponding to the initial ideals with positive real roots.  The positive and really positive parts coincide for Grassmannians with the Pl\"ucker embedding~\cite{SpeyerWilliams} and for linear spaces, but they may differ in general~\cite{VINZANT2012392}. Some sufficient conditions for a point in the positve tropicalization to lie in the really positive tropicalization can be found in \cite{VINZANT2012392, rose2024computing} . Note that our notations are different from those used in ~\cite{SpeyerWilliams}, ~\cite{VINZANT2012392}, and ~\cite{rose2024computing}. 

Our main goal here is to study real and positive tropicalization of varieties of low rank symmetric matrices.  There have been several prior works on the tropical geometry of low rank matrices.  Develin, Santos, and Sturmfels \cite{develin2005rank} defined three notions of ranks for matrices over the tropical semiring.  The Kapranov rank is the one compatible with the tropicalization process described above.

The (complex) tropicalization of the variety of $d\times n$ matrices of rank at most 2 was studied by Markwig and the third author in~\cite{MarkwigYu}.  It coincides with the space of bicolored trees, is shellable, and therefore has nice topology.  Motivated by matrix completion problems, Bernstein described the algebraic matroid of this rank $2$ variety, as well as the skew-symmetric counterpart, in~\cite{Bernstein_rank2}.   Hultman and Jonsson described the topology of the space of tropical matrices of Barvinok rank at most two in~\cite{HultmanJonsson}.  

For symmetric matrices, Cartwright and Chan defined three notions of rank~\cite{CartwrightChan}. Zwick provided two additional notions, the symmetric tropical rank and the symmetric Kapranov rank, and showed that they coincide in the rank $2$ case and in the corank $1$ case.  Again, the tropical Kapranov rank is the one arising from the tropicalization process.  The space of symmetric tropical rank $2$ matrices coincides with the space of symmetric bicolored trees and is shellable~\cite{cai2024symmetric}.  We will see later in Corollary~\ref{cor:symmBarv} a characterization of symmetric Barvinok rank $2$ matrices using the bicolored trees.

All the works mentioned above have been on complex tropicalization.  The positive tropicalization of the space of rank $2$ matrices is described as the space of Barvinok rank 2 matrices by Ardila in~\cite{ardila2004tropical} and Brandenburg, Loho, and Sinn in~\cite{BLS}. (See Remark~\ref{rem:gap}.) The latter paper also gave some criterion for positivity in rank $3$ case and described the positive tropicalization of the determinantal hypersurface.

In this paper we describe the real, positive, and really positive tropicalization of the varieties of \textit{symmetric} $n \times n$ matrices of rank at most $r$, where $r = 2$ or $r = n-1$.  In particular, we compare their real versus complex tropicalizations and positive versus really positive tropicalizations.  We also fill in some gaps in the literature for the general (nonsymmetric) cases.
A summary of our results is shown in Table~\ref{tab:results}.

 \begin{table}[h]
\def\arraystretch{1.5}
    \centering
        \caption{A summary of results}
    \label{tab:results}
    \begin{tabular}{|c|c|c|}
               \hline
                         & ${\mathbf{trop}}_\cR\stackrel{?}{=}\mathbf{trop}_\cC$ & $\mathbf{ trop}_{\cR^+}\stackrel{?}{=}\mathbf{trop}_{\cC^+}$\\
          \hline
    \textbf{ rank 2 } & Yes (\!\!\cite[Theorem 6.5]{develin2005rank}) & Yes (Theorem~\ref{thm:nonsymmrk2}) \\
    \hline
           \textbf{symmetric rank 2 } & Yes (Theorem~\ref{thm:symmrank2real}) & Yes (Theorem~\ref{thm:symmrk2})\\
           \hline
                      \textbf{singular } & Yes (Corollary~\ref{cor:corank1})& Yes (Theorem~\ref{thm:nonsymm-hypersurface-real-positive})\\
           \hline
           \textbf{symmetric singular } & Yes (Theorem~\ref{thm:SingularSymmetric}) & No (Example~\ref{ex:no-real-positive-lift}, Theorem~\ref{thm:tropR+-for-symm-hypersurface}) \\
           \hline
    \end{tabular}
\end{table}
 
\section{Rank 2}

For positive integers $d,n,r$, let $\Mdnr \subset \cC^{d \times n}$ be the variety of $d \times n$ matrices over $\cC$ of rank at most $r$. It is  defined by vanishing of $(r+1) \times (r+1)$ minors of a $d \times n$ matrix of indeterminates.  Our goal is to studay real, positive, and really positive tropicalization of these varieties.

We now recall a few different notions of ranks for tropical matrices introduced in~\cite{develin2005rank}.  For a tropical matrix $A$ in $\RR^{n \times d}$, its \textit{Kapranov rank} over a valued field $\cK$ is defined to be the smallest integer $r$ such that $A$ lies in $\trop_\cK(\Mdnr)$.  
This is equivalent to $A$ being the image under entrywise valuation of a rank $\leq r$ matrix over $\cK$.  In this case we say that $A$ has a rank $r$ {\em lift} to $\cK$.

A tropical matrix $A=(a_{ij})_{1\leq i,j\leq n}$ is \textit{tropically singular} if the minimum in 
\[
\val(\det(A)) = \oplus_{\sigma\in S_n} \left( a_{1\sigma(1)}\odot \cdots \odot a_{n\sigma(n)}\right)  = \min_{\sigma\in S_n} \left( a_{1\sigma(1)}+\cdots+a_{n\sigma(n)} \right)
\]
is attained at least twice. 
The \textit{tropical rank} a matrix $A$ is the size of the largest tropically nonsingular submatrix of $A$.   The \textit{Barvinok rank} of a $d \times n$ tropical matrix $A$ is the smallest integer $r$ such that $A = B \odot C$
 where $B$ is $d\times r$ and $C$ is $r \times n$.  It was shown in~\cite{develin2005rank} that for any matrix we have
 \[
\text{tropical rank} \leq \text{Kapranov rank over $\cK$} \leq \text{Barvinok rank}
 \]
 where the first inequality is valid for any field $\cK$ with nontrivial valuation, and the second inequality is valid when $\cK$ is a Puiseux series field over an infinite field.  The inequalities can be both strict in general.  However
    if a matrix has tropical rank $2$, then it  has Kapranov rank $2$ as well over Puiseux series  with coefficients in an infinite field~\cite[Theorem 6.5]{develin2005rank}.  This immediately implies the following.

\begin{lemma}
    For any $d,n \geq 2$,
        the real and complex tropicalizations coincide for the variety of $d \times n$ rank $2$ matrices:
        \[\trop_\cR(\cM^2_{d,n}) = \trop_\cC(\cM^2_{d,n}). \]
\end{lemma}

 This is not true for rank $3$ matrices.
    \begin{lemma}\label{lem:rnk3}
         The real and complex tropicalizations differ for the variety of $9 \times 12$ matrices of rank at most $3$:
   \[\trop_\cR(\cM^3_{9,12}) \subsetneq \trop_\cC(\cM^3_{9,12}). \]
    \end{lemma}

\begin{proof}
Let us recall the following results from~\cite[Section 7]{develin2005rank}. For a matroid $M$ the cocircuit matrix $C(M)$ has rows indexed by the ground set of $M$ and columns indexed by the cocircuits of $M$. It has a $0$ in entry $(i, j)$ if the $i$-th element is in the $j$-th cocircuit and a $1$ otherwise.  Then the tropical rank of $C(M)$ is the rank of the matroid $M$.  If the Kapranov rank of $C(M)$ over Puiseux series with a ground field $\kk$ is equal to the rank of $M$, then $M$ is representable over $\kk$. If $\kk$ is an infinite field, then the converse also holds.

The ternary affine plane is a rank $3$ matroid on a $9$ element ground set with $12$ cocircuits which is representable over $\CC$ but not over $\RR$~\cite{Oxley}.  
Thus its cocircuit matrix has Kapranov rank~$=3$ over $\cC$ but Kapranov rank $> 3$ over $\cR$.  In other words, it lies in $\trop_\cC(\cM^3_{9,12})$ but not in $\trop_\cR(\cM^3_{9,12})$.
\end{proof}

The space of $d \times n$ matrices of tropical rank $2$ has a simplicial fan structure as the space of \textit{bicolored trees} with $d+n$ leaves, which are trees with $d$ red leaves and $n$ blue leaves such that removing any internal edge partitions the leaves into two parts with both colors in each part  
\cite{MarkwigYu}.
The construction is as follows.
For each $d\times n$ tropical matrix $A$ of tropical rank $2$, the tropical convex hull of its columns, modulo tropical scaling, is a contractible one dimensional polyhedral complex, or a tree. This tree can be made balanced by attaching one infinite ray in each of the coordinate directions $e_1,\dots,e_d$ in a unique way, making it into a tropical line. To obtain the desired bicolored tree, attach a blue leaf labeled $i$ at the position of the $i$-th column of $A$ in the tree, and a red leaf~$j$ to where the infinite ray $e_{j}$ is attached.  The length of each internal edge is the length of the corresponding edge in the tropical convex hull, under the tropical Hilbert metric \[
d(x,y) = \max_{i} (x_i-y_i) - \min_{i} (x_i-y_i).
\] The leaf edges do not have lengths (or are considered to have length $\infty$).  
Conversely, every bicolored metric tree uniquely determines a tropical rank $2$ matrix modulo tropically scaling rows and columns. For example, the following matrix corresponds to the bicolored tree of Figure~\ref{fig:non-symm-non-caterpillar}:
\begin{equation}
\label{eqn:rank2matrix}
    \begin{bmatrix}
    a & 0 & 0 \\
    0 & b & 0 \\
    0 & 0 & c
\end{bmatrix},\ \ a,b,c>0.
\end{equation}
Using the bicolored trees, we will describe the positive tropicalizations of rank two matrices.

\begin{figure}[h]
\begin{tikzpicture}[line cap=round,line join=round,x=1cm,y=1cm, scale=0.5]

\draw [line width=1pt,color=blue] (2,1)-- (3,2);
\draw [line width=1pt, color=red] (2,1) -- (3,0);
\draw [line width=1pt, color=blue] (-2,1) -- (-3,0);
\draw [line width=1pt, color=red] (-2,1) -- (-3,2);
\draw [line width=1pt, color=blue] (0,-2) -- (1,-3);
\draw [line width=1pt, color=red] (0,-2) -- (-1,-3);

\draw [line width=1pt] (0,0)-- (0,-2);
\draw [line width=1pt] (0,0)-- (-2,1);
\draw [line width=1pt] (0,0)-- (2,1);

\draw (-1.5, 0.1) node {$a$};
\draw (1, 1) node {$b$};
\draw (0.3,-1) node {$c$};

\draw[color=blue] (-3.3,0.1) node {$1$};
\draw[color=blue] (3.3,2.1) node {$2$};
\draw[color=blue] (1,-3.5) node {$3$};

\draw[color=red] (-3.3,2.1) node {$1$};
\draw[color=red] (3.3,0.1) node {$2$};
\draw[color=red] (-1,-3.5) node {$3$};

\draw (0, 0.6) node {$O$};
\end{tikzpicture}
\caption{The bicolored tree corresponding to the matrix~\eqref{eqn:rank2matrix}, which has tropical rank $2$ and symmetric tropical rank $3$. Any bicolored tree that is not a caterpillar must contain this as a subtree. The origin is marked as $O$, and $a, b, c$ are positive real numbers.}
\label{fig:non-symm-non-caterpillar}
\end{figure}
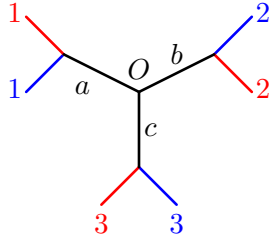


\begin{theorem}
\label{thm:nonsymmrk2}
For any integers $d,n \geq 2$, the positive and really positive tropicalization of $\cM^{2}_{d,n}$ coincide and are equal to the set of $d \times n$ matrices of Barvinok rank at most 2.
\end{theorem}

\begin{proof}
Let $B_{d,n}^2$ be the set of $d\times n$ tropical matrices of Barvinok rank at most $2$.
     We will show that 
    \[
    \tropC(\cM_{d,n}^2) \subseteq B_{d,n}^2 \subseteq \tropR(\cM_{d,n}^2).
    \]
    The reverse containment $\tropR(\cM_{d,n}^2) \subseteq \tropC(\cM_{d,n}^2)$ follows from the definition.

    We first show the inclusion on the right. If $A$ is a $d \times n$ matrix of Barvinok rank 2, then there exists matrices $B \in \RR^{d \times 2}$ and $C \in \RR^{2 \times n}$ satisfying $A = B \odot C$. There are no conditions on~$B$ and~$C$, so any lift of $B$ and  $C$ to $\tilde{B}$ and $ \tilde{C}$ respectively over $\cR^+$ (for instance, simply lifting~$n$ to~$t^n$) induces a matrix $\tilde{A} = \tilde{B}\tilde{C}$. Because matrix multiplication is subtraction-free, and $\tilde{B}$ and $\tilde{C}$ are positive real Puiseux series matrices, no cancellation occurs in the product $\tilde{B}\tilde{C}$ so $\tilde{A}$ has entries in~$\cR^+$ and $\trop(\tilde{A}) = A$.  The matrix $\tilde{A}$ has rank at most $2$ because it is a product of a $d\times 2$ matrix and a $2 \times n$ matrix.

    Now we show that any matrix in $\tropC(\cM_{d,n}^2)$ has Barvinok rank 2. From \cite{MarkwigYu}, we have a bijection between $\cM_{d,n}^2$ modulo lineality, which correpsond to tropically scaling of rows and columns of the matrices, and bicolored metric trees with $d$ leaves of one color and $n$ leaves of the other color. Furthermore, it follows from~\cite[Proposition~6.1]{develin2005rank} that a tropical rank $2$ matrix has Barvinok rank 2 if and only if its bicolored tree is a ``caterpillar'' tree: all of its internal vertices lie along a path. We will now show that any non-caterpillar phylogenetic tree cannot be in $\tropC(M_{d,n}^2)$. 

    If a tree is not caterpillar, then its internal vertices do not all lie along a path. Up to permutation of rows and columns (corresponding to simultaneous permutation of the leaf labels within each color class), it contains Figure \ref{fig:non-symm-non-caterpillar} as a subtree.
By tropically rescaling the rows and columns of the corresponding matrix, which preserves all notions of rank, we obtain the submatrix $$\begin{bmatrix}
    a & 0 & 0\\0 & b & 0\\0 & 0 & c
\end{bmatrix},~~~ a,b,c > 0$$ inside our matrix. The tropical determinant of this submatrix achieves the minimum among the two all-zero monomials. Both of these monomials have positive sign in the $3 \times 3$ determinant polynomial, so plugging in elements of $\cC^+$ does not make the determinant vanish.
Thus any matrix containing this submatrix cannot lie in $\tropC(\cM_{d,n}^2)$, as desired.
\end{proof}

Caterpillar trees also arise as Stiefel tropical lines in~\cite{FinkRincon}.  A matrix of Barvinok rank $2$ gives rise to a Stiefel tropical line as follows.  We can factor a $d \times n$ matrix $A$ of Barvinok rank $2$ as $A = B \odot C$ where $B$ is $d \times 2$ and $C$ is $2 \times n$.  As in the proof above, if we lift $B$ and $C$ to any matrices $\tilde{B}, \tilde{C}$ over $\cR^+$, then there is no cancellation of leading terms in the matrix product $\tilde{B}\tilde{C}$.  So the columns of $A$ lie in the tropicalization of the column space of $\tilde{B}$. Entries of $\tilde{B}$ can be chosen generically enough so that there is no cancellation of leading terms in their $2 \times 2$ minors, then the tropicalization of the column space of $\tilde{B}$ is the Stiefel tropical linear space of the matrix $B$.  In this case the bicolored tree associated to $A$ coincides with the caterpillar tree arising from the Stiefel tropical line, with some additional leaves added to mark the locations of columns of $A$.

\begin{remark}
\label{rem:gap}
The result in Theorem~\ref{thm:nonsymmrk2} has appeared in the literature \cite{ardila2004tropical,BLS}. 
    However the proof in \cite{ardila2004tropical}, which was made more explicit in \cite[Theorem 5.2]{BLS}, relies on \cite[Theorem~2]{PachterSturmfels} for which there is a gap in the proof.  The image of the tropical morphism $g$ in the proof coincides with $\tropR$ of the image variety, not $\tropC$ as claimed in the theorem statement.  Note that $\mathcal{T}^+$ in \cite{PachterSturmfels} refers to our $\tropC$ and the set $\mathcal{R}^+$
 in \cite{SpeyerWilliams} refers to our $\mathcal{C}^+$.

The result is also implied by Corollary~2.14 in \cite{BLS}, but there is a gap in the proof of  Proposition~2.13  that the corollary relies upon, as the proof does not take into account possible cancellation of leading terms.  See the corrigendum~\cite{BLS_corrigendum}.
We do not know of counterexamples to \cite[Theorem~2]{PachterSturmfels} or \cite[Proposition~2.13]{BLS}; we only know that the proofs are incomplete.
\qed
\end{remark}

\section{Symmetric Rank 2}

Let $\Snr \subset \cC^{{n+1 \choose 2}}$ be the variety of $n\times n$ symmetric matrices of rank at most $r$.  The \textit{symmetric Kapranov rank} over $\cK$ of an $n \times n$ symmetric tropical matrix $A$ is defined to be the smallest $r$ such that $A$ lies in $\trop_\cK(\Snr)$, for $\cK = \cC$ or $\cR$.

A symmetric tropical matrix $A$ is \textit{tropically symmetrically singular} if the minimum in $\val(\det(A))= \oplus_{\sigma\in S_n} \left( a_{1\sigma(1)}\odot \cdots \odot a_{n\sigma(n)}\right)  $  is attained at two distinct tropical monomials. 
For example, the tropical determinant of the $3\times 3$ matrix {\tiny $\begin{pmatrix}
    a_{11} & a_{12} & a_{13} \\ a_{12} & a_{22} & a_{23} \\ a_{13} & a_{23} & a_{33}
\end{pmatrix}$} is
\((a_{11}\odot a_{22} \odot a_{33}) \oplus (a_{12} \odot a_{13} \odot a_{23}) \oplus (a_{12} \odot a_{12} \odot a_{33}) \oplus (a_{13} \odot a_{13} \odot a_{22}) \oplus (a_{23} \odot a_{23} \odot a_{11}). \)
Then the classical identity matrix {\tiny$\begin{pmatrix}
    1 & 0 & 0 \\ 0 & 1 & 0 \\ 0 & 0 & 1
\end{pmatrix}$} is tropically singular but tropically symmetrically nonsingular since the minimum in the symmetric determinant is attained uniquely at the tropical monomial $a_{12}\odot a_{13} \odot a_{23}$.  

The \textit{symmetric tropical rank} of a symmetric matrix is the size of the largest tropically nonsingular submatrix of $A$, where we use the notion of the symmetric tropical singularity for the principal submatrices. Equivalently, a symmetric matrix has symmetric tropical rank $\leq r$ if minimum is attained twice in all its $(r+1)\times(r+1)$ tropical minors, where we use the tropicalization of the symmetric determinant for the principal minors.  

The \textit{symmetric Barvinok rank} of an $n\times n$
symmetric matrix $A$ is the smallest integer $r$ such that $A$ can be factored as $A = B \odot B^\top$ where $B$ is an $n \times r$ matrix.  
We have
\[
\text{symmetric tropical rank} \leq \text{symmetric Kapranov rank} \leq \text{ symmetric Barvinok rank}
\]
where the symmetric Kapranov rank is over any valued field, and the inequalities can be strict.  However if a symmetric tropical matrix has symmetric tropical rank $2$, then it has symmetric Kapranov rank $2$ over~$\cC$~\cite[Theorem 6]{zwick2021symmetric} .
We will show that such a matrix has symmetric Kapranov rank $2$ over $\cR$ as well.  We need the following lemma.

\begin{lemma}\label{lem:factoring-discriminant}
    Let $M=(m_{ij})_{1\leq i,j \leq n}$ be an  $n \times n$  symmetric matrix of indeterminates. Fix $i,j \in [n]$ and consider the determinant of $M$ as a quadratic polynomial in $m_{ij}$.
    Then the discriminant of $\det(M)$ with respect to $m_{ij}$ is given by
    \[{\rm Disc}_{m_{ij}}(\det(M))) = 4 M_i M_j\]
    where $M_i$ is the determinant of the submatrix of $M$ obtained by deleting the $i^{\rm th}$ row and the $i^{\rm th}$ column. 
\end{lemma}
\begin{proof}
    Without loss of generality, let $i = 1$ and $j=2$, and assume that $m_{12} = x$. Denote by $M_{S,T}$ the determinant of the submatrix of $M$ obtained by deleting the rows indexed by $S$ and columns indexed by $T$. Let $P(x) = \det(M) = A x^2 + B x + C$. $Ax^2$ consists of all terms that include $m_{12}$ twice, $Bx$ of all terms that include $m_{12}$ once, and $C$ all the terms of the determinant that avoid $m_{12}$. Thus, $A=M_{\{1,2\},\{1,2\}}$, $B=2 (M_{\{2\},\{1\}})|_{x=0}$, and  $C = P(0)$. Thus 
    \[
    {\rm Disc}_{12}(P) = B^2 - 4 A C = 4 (M_{\{2\},\{1\}}|_{x=0})^2 - 4 M_{\{1,2\},\{1,2\}} P(0) = 4( M_1 M_2)|_{x=0}
    \]
    where the last equation is the Dodgson condensation \cite{Dodgson}. Since $M_1$ and $M_2$ are independent of $x$, then the result follows.
\end{proof}

The following lemma and theorem were proved in \cite{zwick2021symmetric} over $\cC$. We now extend them to $\cR$.
\begin{lemma}\label{lem:SubmatricesLifts}
    Let $A$ be an $n\times n$ tropical symmetric matrix of symmetric tropical rank two of the form
    \[
    A = \begin{pmatrix}
        B & \bf{0} & \bf{0}\\
        \bf{0} & 0 & \bf{0} \\
        \bf{0} & \bf{0} & C
    \end{pmatrix}
    \]
    where $B$ is symmetric with positive diagonal entries and $C$ is symmetric with nonnegative entries and having no zero rows or columns. If $\tilde{B}$ and $\tilde{C}$ are symmetric rank $2$ lifts of the submatrices
    \[
    \begin{pmatrix}
        B & \bf{0} \\
        \bf{0} & 0
    \end{pmatrix} \ \ \ \text{ and } \ \ \ \begin{pmatrix}
        0 & \bf{0} \\
        \bf{0} & C
    \end{pmatrix}
    \]
     over $\mathcal{C}$, then $A$ has a symmetric rank $2$ lift $\tilde{A}$ over $\mathcal{C}$. Moreover, if $\tilde{B}$ and $\tilde{C}$ are over $\mathcal{R}$ and $\tilde{C}$ contains a negative $2\times 2$ principal minor, then $A$ has a symmetric rank $2$ lift over $\mathcal{R}$.  
\end{lemma}
\begin{proof}The existence of $\tilde{A}$ over $\cC$ was shown in~\cite{zwick2021symmetric}.  It remains to prove the last statement.  Suppose the symmetric rank $2$ lifts $\tilde{B}, \tilde{C}$ are over~$\cR$, and by permuting rows and columns of $C$ and $ \tilde{C}$ if necessary, we may assume that the top left $2 \times 2$ principal minor of $\tilde{C}$ is negative.

    By scaling the rows and columns we may assume that the bottom right entry of $\tilde{B}$ and the top left entry of $\tilde{C}$ are equal to one.
    Let $m$ and $r$ denote the dimensions of $B$ and $C$ respectively. Thus $n = m+r+1$. Consider the principal submatrix $M = A[m,m+1,m+2]$ of $A$ where $A[S]$ denote the submatrix of $A$ indexed on the rows and columns by $S$. Thus
    \[ M=
     \begin{pmatrix}
     a_{m , m} & 0 & 0 \\ 0 & 0 & 0 \\ 0 & 0 & a_{m+2 , m+2}
    \end{pmatrix}
    \]
    with $a_{m,m}$ is the bottom right entry of $B$ and $a_{m+2,m+2}$ is the top left entry of $C$. Using the lifts $\tilde{B}$ and $\tilde{C}$ we let
    \[
    \tilde{M} = 
     \begin{pmatrix}
     \tilde{a}_{m , m} & \tilde{a}_{m , m+1} & x \\ \tilde{a}_{m , m+1} & 1 & \tilde{a}_{m+1 , m+2} \\ x & \tilde{a}_{m+1 , m+2} & \tilde{a}_{m+2 , m+2}
    \end{pmatrix}
    \]
    where $x$ is an indeterminate.
    We want this matrix to be symmetric and singular.  Setting its determinant to zero gives a quadratic equation $x^2 + bx + c = 0$. The term $bx$ contains $2\tilde{a}_{m, m+1}\tilde{a}_{m+1, m+2}x$ as the element with smallest valuation, so $b$ has valuation zero. The term $c$ has valuation as least zero, since $a_{m,m}>0$ and $a_{m+2,m+2} \ge 0$.
   The product of the two roots is $c$, and the sum of the two roots is $b$, so one of the roots must have valuation zero and the other has nonnegative valuation. There is no cancellation of the leading terms when we add the roots; otherwise both roots would have negative valuations, but their product has nonnegative valuation. 
    
    Let $x$ be a root with valuation zero. For $x$ to be in $\mathcal{R}$, we want the discriminant of this quadratic equation to be nonnegative. If the discriminant is zero, then $x\in \mathcal{R}$. Suppose the discriminant is nonzero. Using Lemma~\ref{lem:factoring-discriminant}, we want the determinants of the $2\times 2$ submatrices $\tilde{M}[1,2]$ and $\tilde{M}[2,3]$ to have the same sign. 
    The submatrix $\tilde{M}[2,3]$ is the top left principal $2 \times 2$ submatrix of $\tilde{C}$, so by our assumption  $\det \tilde{M}[2,3]\leq 0$. And we have $\det \tilde{M}[1,2]= \tilde{a}_{m,m}-\tilde{a}_{m,m+1}^2$. Since the valuation of $\tilde{a}_{m,m}$ is positive and the valuation of $\tilde{a}_{m,m+1}$ is zero, so the leading term of $\det \tilde{M}[1,2]<0$ comes from $-\tilde{a}_{m,m+1}^2$ and is negative. This implies that the discriminant is positive. Thus, we get a symmetric rank $2$ lift $\tilde{M}$ of $M$ over $\mathcal{R}$. 
    
    The remaining entries of the matrix can then be completed using linear combinations of existing rows and columns, as in the proof of Lemma~3 in \cite{zwick2021symmetric}.

\end{proof}
\begin{theorem}
\label{thm:symmrank2real}
    Let $A$ be an $n\times n$ tropical symmetric matrix. Then $A$ has symmetric Kapranov rank two over $\mathcal{R}$ if and only if it has symmetric tropical rank two.  
\end{theorem}
\begin{proof}
    We will follow the proof of Theorem $6$ in\cite{zwick2021symmetric}. While Zwick proved the theorem over $\mathcal{C}$, most parts of the proof also work over $\mathcal{R}$, except for some details that we will identify and adjust to make the proof valid over $\mathcal{R}$.
    
    The forward direction $(\Rightarrow)$ follows from the facts that the symmetric tropical rank is at most equal to the symmetric Kapranov rank and that if $A$ has symmetric tropical rank one, then we can find a symmetric lift of $A$ over $\mathcal{R}$ with standard rank one.
    
    For the other direction $(\Leftarrow)$, the main idea of the proof is that any symmetric matrix $A$ with symmetric tropical rank two can be written in the following form, after possibly a tropical scaling and diagonal permutation:
    \[
    \begin{pmatrix}
        B_1 & \bf{0} & \bf{0} & \bf{0} & \bf{0} \\
        \bf{0} & B_2 & \bf{0} & \bf{0} & \bf{0} \\
        \bf{0} & \bf{0} & \bf{0} & \bf{0} & \bf{0} \\
        \bf{0} & \bf{0} & \bf{0} & \bf{0} & C\\
        \bf{0} & \bf{0} & \bf{0} & C^\top & \bf{0}
    \end{pmatrix}
    \]
    the matrices $B_1$ and $B_2$ are symmetric and positive, and the matrix $C$ is non-negative and has no zero columns. The matrix $C^\top$ represents the transpose of $C$ and each $\bf{0}$ represents a zero matrix of the appropriate size.   The matrix $A$ does not need to contain any zero rows or columns, and some of the blocks $B_1$, $B_2$, and $C$ may be empty.  But we may assume that $A$ does not contain only one positive block and that it is not the zero matrix, \cite[Lemma~4]{zwick2021symmetric}. 
    
    We first assume that $A$ contains exactly one zero row/column. Using Lemmas $1$ and $2$ in \cite{zwick2021symmetric}, we construct symmetric rank $2$ lifts  to $\cR$ for the two submatrices:
    \[ A_1 =
    \begin{pmatrix}
        B_1 & \bf{0} & \bf{0}\\ \bf{0} & B_2 & \bf{0}\\ \bf{0} & \bf{0} & 0
    \end{pmatrix}
    \ \ \text{ and } \ \ A_2 =\begin{pmatrix}
        0 & \bf{0} & \bf{0} \\ \bf{0} & \bf{0} & C \\ \bf{0} & C^\top & \bf{0}
    \end{pmatrix}
    \]
    denoted by $\tilde{A_1}$ and $\tilde{A_2}$, respectively. We can adjust the proof of Lemma $1$ in \cite{zwick2021symmetric} to ensure that $\tilde{A_2}$ contains a $2\times 2$ principal submatrix with negative determinant (the construction of the lift is done by choosing generic entries, we require one of these generic choices satisfies this inequality, so it is still generic). Now using Lemma~\ref{lem:SubmatricesLifts}, we conclude that $A$ has the desired lift $\tilde{A}$. 

    If $A$ contains more than one zero row/columns, then we proceed by induction, where the base case is the covered by the above discussion. Finally, if $A$ contains no zero row/column, then the matrix $\begin{pmatrix}
        0 & \bf{0} \\ \bf{0} & A
    \end{pmatrix}$ has symmetric tropical rank two by \cite[Lemma~5]{zwick2021symmetric}, and therefore symmetric Kapranov rank two by the previous discussion. By eliminating
    the first row/column from the lift we get a symmetric rank $2$ lift of $A$ over $\mathcal{R}$. Therefore $A$ has symmetric Kapranov rank two as well.
\end{proof}

The theorem above, combined with \cite[Theorem~6]{zwick2021symmetric}, which says that symmetric tropical rank~$2$ matrices have symmetric Kapranov rank~$2$ over $\cC$, gives us the following.
\begin{corollary}
     The real and complex tropicalizations coincide for symmetric matrices of rank at most $2$. That is,
    \[
    \trop_\cR(\cS^2_n) = \trop_\cC(\cS^2_n).
    \]
\end{corollary}

Now let us consider the positive and really positive parts of $\mathcal{S}_n^2$. As in the nonsymmetric case, we will use the notion of bicolored trees to classify the possible cases. Inspired by the case of nonsymmetric rank 2 matrices, Cai, Lee, and Yu \cite{cai2024symmetric} proved that the polyhedral fan $\cS^2_n$ has a simplicial fan structure given by a special type of bicolored trees called symmetric bicolored trees, or a \textit{symbic trees}. The symbic trees are bicolored trees on $n$ blue leaves and $n$ red leaves which are invariant under swapping the colors, with the additional assumption that color-swapping the leaves induces an automorphism on the whole tree whose fixed points form a path (possibly a single point).

For example, the bicolored tree in Figure~\ref{fig:non-symm-non-caterpillar} is not symbic as its fixed points do not form a path, although it is symmetric under swapping colors.  
On the other hand, the matrix
\[
A = \begin{pmatrix}
    0 & a & b \\
    a & 0 & 0\\ 
    b & 0 & 0
\end{pmatrix},
\ \ \ a,b>0
\]
has symmetric rank $2$, and its  symbic tree is shown in Figure~\ref{fig:caterpillar-symbic-tree-3}(a). See \cite{cai2024symmetric} for more details. 

\begin{theorem}\label{thm:symmrk2}
    The positive and really positive parts of the tropical variety $\cS^2_n$ coincide, and are equal to the set of $n \times n$ symmetric matrices of Barvinok rank at most 2.
\end{theorem}
Here we are using the usual Barvinok rank and not the symmetric Barvinok rank; that is, we do not require the tropical matrix factorization to be symmetric.  Symmetric Barvinok rank $2$ matrices will be discussed in Corollary~\ref{cor:symmBarv}.


\begin{proof}Let ${\it SB}^2_n$ be the set of symmetric tropical matrices with Barvinok rank at most $2$.
    As in the proof of Theorem~\ref{thm:nonsymmrk2}, since $\tropR(\cS^2_n) \subseteq \tropC(\cS^2_n)$, we proceed by proving that
    \[
  \tropC(\cS^2_n) \subseteq {{\it SB}}^2_n \subseteq \tropR(\cS^2_n).
    \]
        
We will first prove the inclusion on the right.  We cannot use the Barvinok rank 2 factorization directly as we did in the proof of Theorem~\ref{thm:nonsymmrk2} because the factorization is not always symmetric. Instead we will construct explicit lifts to rank $2$ matrices over $\cR^+$ for a symmetric matrix with Barvinok rank 2.
Having Barvinok rank 2 still is equivalent to the bicolored trees being caterpillar. There are two types of caterpillar symbic trees, shown in Figure~\ref{fig:Barv-rk-2-symbic}.

    First we construct an explicit lift for matrices of the type in Figure \ref{fig:Barv-rk-2-symbic}(a). For this tree type, we denote by $i$ the internal vertex adjacent to leaves $i$ and $i'$. As matrices preserve symmetric tropical rank and Barvinok rank under simultaneous permutations of rows and columns (and their corresponding trees simultaneously permute the leaves of each color), we can consider without loss of generality that the leaves are ordered $1, n, \ldots, 3, 2$.
    We denote by $d_i$ the length along the internal vertices from $1$ to $i$. In particular, we have $0 \le d_n \le d_{n-1} \le \ldots \le d_2$.
    Up to simultaneous tropical scaling of rows and columns, this caterpillar symbic tree corresponds to a unique matrix $$ M = \begin{bmatrix}0 & 0 & 0 &\ldots& 0 \\
    0 & d_2 & d_3 & \ldots & d_n \\
    0 & d_3 & d_3 & \ldots & d_n \\
    \vdots & \vdots & \vdots & \vdots & \vdots\\
    0 & d_n & d_n & \ldots & d_n\end{bmatrix}.$$
    In particular, $$M_{ij} = \begin{cases}
        0 & \text{if } i=1 \text{ or } j=1\\
        d_{\max(i,j)} & \text{otherwise.}
    \end{cases}.$$
    This matrix has explicit lift $$\tilde{M} = \begin{bmatrix} 1 & 1 & 1+t^{d_3} & \ldots & 1+t^{d_n} \\
    1 & t^{d_2} & t^{d_3} + t^{d_2} & \ldots & t^{d_n} + t^{d_2} \\
    1+t^{d_3} & t^{d_3} + t^{d_2} & 2t^{d_3} + t^{2d_3} + t^{d_2} & \ldots & t^{d_n} + t^{d_3} + t^{d_n + d_3} + t^{d_2} \\
    \vdots & \vdots & \vdots & \vdots & \vdots\\
    1 + t^{d^n} & t^{d_n} + t^{d_2} & t^{d_n} + t^{d_3} + t^{d_n + d_3} + t^{d_2} & \ldots & 2t^{d_n} + t^{2d_n} + t^{d_2}
    \end{bmatrix}.$$
    In particular, we choose $$\tilde{M}_{ij} = \begin{cases} 
    1 & \text{if } i=j=1\text{, or }i=2, j=1\\
    t^{d_2} & \text{if } i=j=2\\
    t^{d_i}\cdot\tilde{M}_{1j} + \tilde{M}_{2j} & \text{if } i\ge j, i > 2 \\
    \tilde{M}_{ji} & \text{if } i < j.
    \end{cases}.$$
    In the matrix $\tilde{M}$ above, the leftmost term is the leading term in each entry.
    By construction this matrix is symmetric. Furthermore, each row is a positive linear combination of the (positive) first two rows, so it is of usual rank 2 and every entry is positive. It remains to show that $\trop(\tilde{M}) = M$.

    For $i, j \le 2$ by inspection the matrix has the right valuation. For $i > 2$, then entries of the first column
    \[\tilde{M}_{i1} = t^{d_i} \tilde{M}_{11} + \tilde{M}_{21} = t^{d_i} + 1\]
    have valuation zero, and the entries of the second column
    \[\tilde{M}_{i2} = t^{d_i} \tilde{M}_{12} + \tilde{M}_{22} = t^{d_i} \cdot 1 + t^{d_2} = t^{d_i} + t^{d_2}\]
    have valuations $d_i$ as desired. By the symmetric nature of the matrix, the first two rows also have the correct valuation. Then for $i > 2, j > 2$, $\tilde{M}_{ij} = t^{d_i}\tilde{M}_{i1} + \tilde{M}_{2i}$. We cannot have cancellation because everything is positive, so $$\text{val}(\tilde{M}_{ij}) = \min(\text{val}(t^{d_i}\tilde{M}_{1j}),\text{val}(\tilde{M}_{2j})) = \min(d_i + 0, d_j)  = d_{\max(i,j)}$$ as desired.

    Now we show that the second type of tree also has a lift. The second type, in Figure~\ref{fig:Barv-rk-2-symbic}(b), is not unique up to simultaneous row and column permutation for $n > 3$, so we will have to examine it more closely. For $n=3$, the only tree up to simultaneous row and column permutation is Figure~\ref{fig:caterpillar-symbic-tree-3}(a). 
    Let $d_i$ be the distance from the fixed point $O$ to the internal vertex attached to the leaf $i$. We have $0 \le d_3 \le d_2 = d_1$. Then the matrix associated with Figure~\ref{fig:caterpillar-symbic-tree-3}(a) is 
    $$M = \begin{bmatrix} 0 & d_2 & d_3 \\ d_2 & 0 & 0 \\ d_3 & 0 & 0\end{bmatrix}.$$
 Then $M$ has a symmetric factorization into $M = M_1 \odot M_1^\top$ where $M_1 = \begin{bmatrix} 0 & d_2 \\ d_2 & 0\\d_3 & 0\end{bmatrix}$. Any lift of $M_1$ into $\cR^+$ gives us a  symmetric rank $2$ lift of $M$ to $\cR^+$. More concretely, we can take 
    $$\tilde{M_1} = \begin{bmatrix}1 & t^{d_2} \\ t^{d_2} & 1\\t^{d_3} & 1\end{bmatrix}, \; \tilde{M} = \tilde{M_1} \odot \tilde{M_1}^\top = \begin{bmatrix} 1+t^{2d_2} & 2t^{d_2} & t^{d_3} + t^{d_2}\\2t^{d_2} & 1+t^{2d_2} & 1+t^{d_2+d_3}\\t^{d_3} + t^{d_2}& 1+t^{2d_2} & 1+t^{2d_3}\end{bmatrix}.$$

    For $n > 3$, 
    we can assume that, after simultaneous row and column permutation, as in Figure~\ref{fig:caterpillar-symbic-tree-3} the outer two leaves are $1$ and $2$, and that $2$ and $3$ share the same color on each ``side'' of the tree, and the other leaves are $4, \ldots, n$ in order from outside in. Then, up to these assumptions, the only kinds of symbic trees of this type are determined entirely by whether or not the color of leaf $i$, $i>3$, shares the color of the leaf $1$ or the leaf $2$ on each side.

    Defining distances as in the $n=3$ case, a tree of this type has corresponding matrix 
    
    $$M_{ij} = \begin{cases} 0 & \text{if the leaves } i, j \text{ have the same color on each side of the tree},\\
    \min(d_i, d_j) & \text{otherwise.}\end{cases}$$

    We claim that every tree of this type has a matrix $M$ with a symmetric factorization into $M_1 \odot M_1^\top$. In particular, the $i$-th row of $M_1$ is $\begin{bmatrix}0 & d_i\end{bmatrix}$ if leaf $i$ and leaf $1$ share the same color on each side of the tree, and $\begin{bmatrix}d_i & 0\end{bmatrix}$ otherwise. If $i$ and $j$ share the same color on each side of the tree, then without loss of generality $M_{ij} = \begin{bmatrix} 0 & d_i\end{bmatrix} \odot \begin{bmatrix} 0 \\ d_j\end{bmatrix} = \min(0, d_i+d_j) = 0$, and if they have different colors, then without loss of generality $M_{ij} = \begin{bmatrix} d_i & 0 \end{bmatrix} \odot \begin{bmatrix} 0 \\ d_j \end{bmatrix} = \min(d_i, d_j)$ as desired. Thus, as in the $n=3$ case, any positive lift of $M_1$ yields a symmetric rank 2 positive lift of $M$.

    For example, the trees in Figure \ref{fig:caterpillar-symbic-tree-3}(b) and Figure \ref{fig:caterpillar-symbic-tree-3}(c) have, respectively, tropical symmetric matrix factorizations into $$M_{1(b)} = \begin{bmatrix}0 & d_1\\d_2 & 0 \\ d_3 & 0\\d_4 & 0\end{bmatrix} \text{ and } M_{1(c)} = \begin{bmatrix}0 & d_1 \\ d_2 & 0 \\ d_3 & 0 \\ 0 & d_4\end{bmatrix}.$$

    Now that we have shown that the set of Barvinok rank at most 2 symmetric matrices are contained in $\tropR (\cS_{n}^2)$, we need to show that $\tropC (\cS_n^2)$ is contained in the set of Barvinok rank 2 matrices and then we are done. As in the proof of Theorem~\ref{thm:nonsymmrk2}, any matrix that isn't Barvinok rank 2 must contain the matrix associated with the tree in Figure~\ref{fig:non-symm-non-caterpillar} as a subtree. For symmetric matrices, this must appear as some \textit{non}simultaneous permutation of rows and columns, but it still yields a tropical subdeterminant where the minimum is achieved at two monomials of the same sign, and so any matrix that is not Barvinok rank 2 must not lie in $\tropC(\cS_n^2)$, as desired. 
\end{proof}

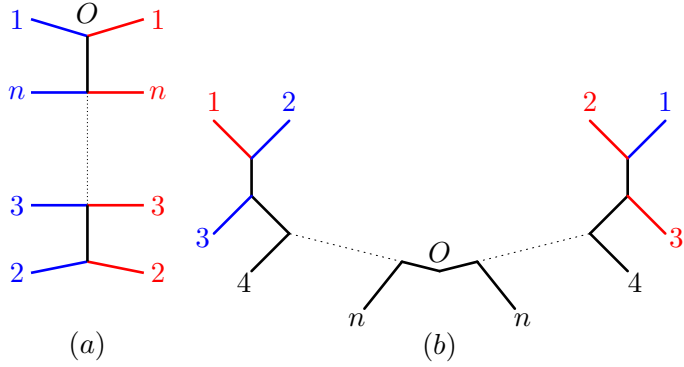
\begin{figure}
    \centering
    \begin{tikzpicture}[scale=0.75]
        \draw [line width=1pt, blue] (0, 3) -- (-1, 3.3);
        \draw [line width=1pt, blue] (0,2) -- (-1, 2);
        \draw [line width=1pt, blue] (0, 0) -- (-1, 0);
        \draw [line width=1pt, blue] (0, -1) -- (-1, -1.2);
        
        \draw [line width=1pt, red] (0,3) -- (1, 3.3);
        \draw [line width=1pt, red] (0,2) -- (1, 2);
        \draw [line width=1pt, red] (0, 0) -- (1, 0);
        \draw [line width=1pt, red] (0, -1) -- (1, -1.2);

        \draw [line width=1pt] (0, 3) -- (0,2);
        \draw [densely dotted] (0,2) -- (0, 0);
        \draw [line width=1pt] (0,0) -- (0, -1);

        \draw[color=blue] (-1.25, 3.3) node {$1$};
        \draw[color=blue] (-1.25, 2) node {$n$};
        \draw[color=blue] (-1.25, 0) node {$3$};
        \draw[color=blue] (-1.25, -1.2) node {$2$};

        \draw[color=red] (1.25, 3.3) node {$1$};
        \draw[color=red] (1.25, 2) node {$n$};
        \draw[color=red] (1.25, 0) node {$3$};
        \draw[color=red] (1.25, -1.2) node {$2$};

        \draw (0, 3.4) node {$O$};
        \draw (0, -2.5) node {$(a)$};
    \end{tikzpicture}
    \begin{tikzpicture}[line cap=round,line join=round,x=1cm,y=1cm, scale=0.5]
\draw [line width=1pt] (0,-1)-- (1, -0.75);
\draw [dotted] (1, -0.75) -- (4,0);
\draw [line width=1pt] (0,-1)-- (-1, -0.75);
\draw [dotted] (-1, -0.75) -- (-4,0);
\draw [line width=1pt] (4,0)-- (5,1);
\draw [line width=1pt] (-4,0)-- (-5,1);
\draw [line width=1pt] (5,1)-- (5,2);
\draw [line width=1pt] (-5,1)-- (-5,2);

\draw [line width=1pt] (-4,0) -- (-5,-1);
\draw [line width=1pt] (-1, -0.75) -- (-2, -2);
\draw [line width=1pt, color=blue] (-5,1) -- (-6, 0);
\draw [line width=1pt, color=red] (-5,2) -- (-6,3);
\draw [line width=1pt, color=blue] (-5,2) -- (-4,3);

\draw [line width=1pt] (4,0) -- (5,-1);
\draw [line width=1pt] (1, -0.75) -- (2, -2);
\draw [line width=1pt, color=red] (5,1) -- (6, 0);
\draw [line width=1pt, color=blue] (5,2) -- (6,3);
\draw [line width=1pt, color=red] (5,2) -- (4,3);

\draw[color=blue] (-4,3.5) node {$2$};
\draw[color=red] (-6,3.5) node {$1$};
\draw[color=blue] (-6.3, -0.1) node {$3$};
\draw (-5.2, -1.3) node {$4$};
\draw (-2.2, -2.3) node {$n$};

\draw[color=red] (4,3.5) node {$2$};
\draw[color=blue] (6,3.5) node {$1$};
\draw[color=red] (6.3, -0.1) node {$3$};
\draw (5.2, -1.3) node {$4$};
\draw (2.2, -2.3) node {$n$};

\draw (0, -0.5) node {$O$};

\draw (0, -3) node {$(b)$};
\end{tikzpicture}
    \caption{The two types of Barvinok rank 2 symmetric tropical matrices, as used in the proof of Theorem~\ref{thm:symmrk2}. In type (b), the black leaves $4, \ldots, n$ have indeterminate, opposite colors---every choice of color pair determines a different combinatorial type. We suppress edge lengths. In (a), internal edges can be any nonnegative number, while in (b) edge lengths must be symmetric about the midpoint. Without loss of generality, we translate the trees so that $O$ marks the origin.}
    \label{fig:Barv-rk-2-symbic}
\end{figure}

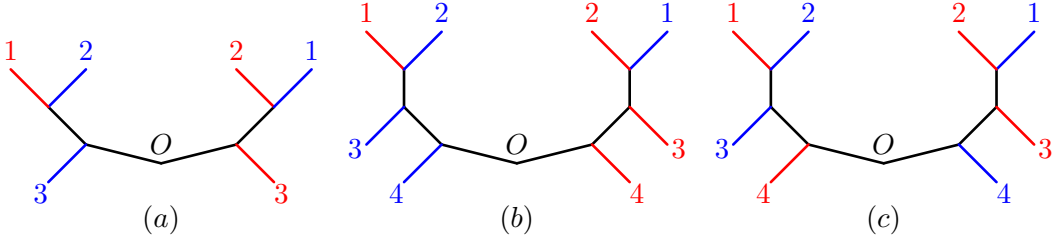
\begin{figure}
    \centering
    \begin{tikzpicture}[line cap=round,line join=round,x=1cm,y=1cm, scale=0.5]

\draw [line width=1pt] (0,-0.5)-- (2,0);
\draw [line width=1pt] (0,-0.5)-- (-2,0);
\draw [line width=1pt] (2,0)-- (3,1);
\draw [line width=1pt] (-2,0)-- (-3,1);

\draw [line width=1pt, color=blue] (-2,0) -- (-3,-1);
\draw [line width=1pt, color=blue] (-3,1) -- (-2, 2);
\draw [line width=1pt, color=red] (-3,1) -- (-4,2);

\draw [line width=1pt, color=red] (2,0) -- (3,-1);
\draw [line width=1pt, color=red] (3,1) -- (2, 2);
\draw [line width=1pt, color=blue] (3,1) -- (4,2);

\draw[color=red] (-4,2.5) node {$1$};
\draw[color=blue] (-2, 2.5) node {$2$};
\draw[color=blue] (-3.2, -1.3) node {$3$};

\draw[color=blue] (4,2.5) node {$1$};
\draw[color=red] (2, 2.5) node {$2$};
\draw[color=red] (3.2, -1.3) node {$3$};

\draw (0, 0) node {$O$};

\draw (0, -2) node {$(a)$};
\end{tikzpicture}
\begin{tikzpicture}[line cap=round,line join=round,x=1cm,y=1cm, scale=0.5]
\draw [line width=1pt] (0,-0.5)-- (2,0);
\draw [line width=1pt] (0,-0.5)-- (-2,0);
\draw [line width=1pt] (2,0)-- (3,1);
\draw [line width=1pt] (-2,0)-- (-3,1);
\draw [line width=1pt] (3,1)-- (3,2);
\draw [line width=1pt] (-3,1)-- (-3,2);

\draw [line width=1pt, color=blue] (-2,0) -- (-3,-1);
\draw [line width=1pt, color=blue] (-3,1) -- (-4, 0);
\draw [line width=1pt, color=red] (-3,2) -- (-4,3);
\draw [line width=1pt, color=blue] (-3,2) -- (-2,3);

\draw [line width=1pt, color=red] (2,0) -- (3,-1);
\draw [line width=1pt, color=red] (3,1) -- (4, 0);
\draw [line width=1pt, color=blue] (3,2) -- (4,3);
\draw [line width=1pt, color=red] (3,2) -- (2,3);

\draw[color=blue] (-2,3.5) node {$2$};
\draw[color=red] (-4,3.5) node {$1$};
\draw[color=blue] (-4.3, -0.1) node {$3$};
\draw[color=blue] (-3.2, -1.3) node {$4$};

\draw[color=red] (2,3.5) node {$2$};
\draw[color=blue] (4,3.5) node {$1$};
\draw[color=red] (4.3, -0.1) node {$3$};
\draw[color=red] (3.2, -1.3) node {$4$};

\draw (0, 0) node {$O$};

\draw (0, -2) node {$(b)$};
\end{tikzpicture}
\begin{tikzpicture}[line cap=round,line join=round,x=1cm,y=1cm, scale=0.5]
\draw [line width=1pt] (0,-0.5)-- (2,0);
\draw [line width=1pt] (0,-0.5)-- (-2,0);
\draw [line width=1pt] (2,0)-- (3,1);
\draw [line width=1pt] (-2,0)-- (-3,1);
\draw [line width=1pt] (3,1)-- (3,2);
\draw [line width=1pt] (-3,1)-- (-3,2);

\draw [line width=1pt, color=red] (-2,0) -- (-3,-1);
\draw [line width=1pt, color=blue] (-3,1) -- (-4, 0);
\draw [line width=1pt, color=red] (-3,2) -- (-4,3);
\draw [line width=1pt, color=blue] (-3,2) -- (-2,3);

\draw [line width=1pt, color=blue] (2,0) -- (3,-1);
\draw [line width=1pt, color=red] (3,1) -- (4, 0);
\draw [line width=1pt, color=blue] (3,2) -- (4,3);
\draw [line width=1pt, color=red] (3,2) -- (2,3);

\draw[color=blue] (-2,3.5) node {$2$};
\draw[color=red] (-4,3.5) node {$1$};
\draw[color=blue] (-4.3, -0.1) node {$3$};
\draw[color=red] (-3.2, -1.3) node {$4$};

\draw[color=red] (2,3.5) node {$2$};
\draw[color=blue] (4,3.5) node {$1$};
\draw[color=red] (4.3, -0.1) node {$3$};
\draw[color=blue] (3.2, -1.3) node {$4$};

\draw (0, 0) node {$O$};

\draw (0, -2) node {$(c)$};
\end{tikzpicture}
    \caption{(a) is the only symbic tree of its combinatorial type (i.e., up to simultaneous permutations of red and blue leaves) for $n=3$. (b) and (c) are the only two symbic trees of this combinatorial type for $n=4$. Each tree is translated so that the origin is at the point marked $O$.}
    \label{fig:caterpillar-symbic-tree-3}
\end{figure}

This result can be restated using the tropical generators introduced in \cite{BLS}.
\begin{corollary}\label{cor:symm-rk2-positive-generators}
    The $3 \times 3$ minors of the symmetric $n \times n$ matrix forms a set of positive and really-positive tropical generators for $\cS_n^2$. 
\end{corollary}

In other words, it suffices to check that the minimum in every $3 \times 3$ tropical minor is achieved at two monomials of opposite signs to see that a $n \times n$ symmetric rank 2 matrix is in the positive, or really-positive, parts of the variety.

\begin{proof}
    From the proof of Theorem~\ref{thm:symmrk2}, if a matrix has Barvinok rank $> 2$, then it has some $3 \times 3$ submatrix where the minimums in its tropical determinant are achieved at only monomials of the same sign, which is the contrapositive of the corollary.
\end{proof}

\begin{corollary}
\label{cor:symmBarv}
    A symmetric tropical rank $2$ matrix has symmetric Barvinok rank $2$ if and only if its symbic tree is caterpillar with only one fixed point.
\end{corollary}

\begin{proof}
    The ``if'' direction follows from the proof of the theorem above, where an explicit symmetric factorization is given.
    The other direction goes as follows. Let $A$ be a $n \times n$ symmetric tropical rank 2 matrix, and assume we have a symmetric factorization $A = B \odot B^\top$, where $B$ is an $n \times 2$ matrix. 
 By permuting and tropically rescaling rows and columns of $A$ and $B$,
    we may assume that $B$ has all nonnegative entries, with at least one 0 in each row. We may assume $B$ does not have a zero column; otherwise $B \odot B$ would be the zero matrix. Thus, up to row and column permutations, the first three rows of $B$ appear as $\begin{bmatrix} 0 & b_1 \\ b_2 & 0\\ b_3 & 0\end{bmatrix}$, and the remaining rows are either $\begin{bmatrix} 0 & b_i\end{bmatrix}$ or $\begin{bmatrix} b_i & 0\end{bmatrix}$ with $b_i \geq 0$.
    
    Furthermore, we can assume that $b_3 \le b_1$ and $b_3 \le b_2$, and for $i \ge 4$, $b_i \le b_{i-1}$. Note that $b_1$ or $b_2$ can only appear in $A = B \odot B^\top$ at $A_{1,2}$ and $A_{2,1}$, and in particular $A_{1,2} = A_{2,1} = \min(b_1, b_2)$. Thus we may replace $B$ with $B'$ by replacing $b_1, b_2$ with $\min(b_1, b_2)$, without affecting $A$. This $B'$ is precisely the matrix factorization of a matrix whose symbic tree is a caterpillar of only one fixed point from the proof of Theorem~\ref{thm:symmrk2}.
    
    In other words, from following the proof of Theorem~\ref{thm:symmrk2} backwards, any symmetric Barvinok rank 2 matrix has symbic trees as in Figure~\ref{fig:caterpillar-symbic-tree-3}, with caterpillars of only one fixed point, and leaf colors determined by whether row $i$ has the zero on the first or second column.
\end{proof}

\section{Singular matrices}
\label{sec:square-corank-one}

The variety $\cM_{n,n}^{n-1}$ of singular $n\times n$ matrices is a hypersurface defined by the determinant of an $n \times n$ matrix of indeterminates.
The \textit{Newton polytope} of a polynomial $f$ is the convex hull of the exponent vectors of its monomial terms. The (inner) \textit{normal cone} at a point $x$ of a convex set $S\subset \mathbb{R}^n$ is the set of real vectors $\{v \in \mathbb{R}^n: x\cdot v \leq y\cdot v~ \forall y \in S\}$.  For a polynomial $f$ with trivially valued coefficients, the (complex) tropicalization of its hypersurface consists of normal cones to the edges of the Newton polytope.  

The Newton polytope of the determinant is the Birkhoff polytope, whose vertices are $n \times n$ permutation matrices. It consists of matrices with nonnegative entries, all of whose row sums and column sums are equal to 1.  The tropical hypersurface $\trop_\cC(\cM_{n,n}^{n-1})$ is the codimension-1 skeleton of the normal fan of the Birkhoff polytope. Its maximal cones are normal cones of the edges of the polytope.  This hypersurface has lineality space spanned by matrices which are everywhere zero except along a row or column of all ones. Two permutations $\sigma_1$ and $ \sigma_2 \in S_n$ form an edge in the Birkhoff polytope when $\sigma_1\sigma_2^{-1}$ is a cycle \cite[Corollary 2.1]{BL96}. A maximal cone of $\trop_\cC(\cM_{n,n}^{n-1})$ dual to the edge formed by $\sigma_1$ and $\sigma_2$ is in the positive part $\trop_{\cC^+}(\cM_{n,n}^{n-1})$ exactly when $\sigma_1$ and $\sigma_2$ have opposite signs \cite[Proposition 3.1]{BLS}. 


\begin{theorem}\label{thm:nonsymm-hypersurface-real-positive}
    If the minimum of the tropical determinant of a $n \times n$ matrix $M$ is achieved at two monomials $\sigma_1$ and $ \sigma_2$ with opposite signs, then that matrix has a lift to an $n \times n$ singular matrix of real positive Puiseux series entries. That is, \[\tropC(\cM^{n-1}_{n,n}) = \tropR(\cM^{n-1}_{n,n}).\]
\end{theorem}

This follows from \cite[Lemma 2.6]{VINZANT2012392}. But we give a constructive proof here.
    
\begin{proof}
    Since the really positive part of the tropical variety $\tropR(\cM_{n,n}^{n-1})$ is a closed subset of the tropical variety \cite[Theorem 4.1]{allamigeon2020tropical},  it suffices to consider only the interiors of maximal cones and $\tropC(\cM_{n,n}^{n-1})$ and
    assume that the minimum in the tropical determinant is achieved exactly at two distinct permutations $\sigma_1$ and $\sigma_2$. 
    
    By the lineality space of the tropical hypersurface, we can add multiples of the all-ones vector to any column, so without loss of generality we may assume that $M_{i\sigma_1(i)} = 0$ for every $i \in [n]$. Thus the tropical monomial $\sigma_1$ has valuation 0, and because the minimum is achieved at $\sigma_1$ and $\sigma_2$, $\sigma_2$ also has valuation 0 and any other monomial in the tropical hypersurface must be strictly positive.
    
    Because $\sigma_1$ and $\sigma_2$ are different monomials, $\sigma_1(i) \neq \sigma_2(i)$ for some $i$. Let us make a preliminary lift of $M$ to $\tilde{M} \in \cR_+^{n \times n}$ by arbitrarily choosing any positive real Puiseux series of the right valuation for every entry except the $(i,\sigma_1(i))$ entry, and let
    the $(i,\sigma_1(i))$ entry be an indeterminate variable~$x$.
    We need $\tilde{M}$ to be a singular matrix, so we need  $$\det(\tilde{M}) = Ax + B = 0.$$
It remains show that $x = -B/A$ is a positive real Puiseux series of valuation~$0$.

The determinant of $\tilde{M}$ is the sum of all permutation monomials. Thus $B$ consists of the sum of the monomials that do not contain $\tilde{M}_{i\sigma_1(i)}$. By construction, this contains the monomial $\sigma_2$, which has valuation 0. Every other monomial that it contains has strictly higher valuation, so cancellation cannot occur and the leading term of $B$ is precisely the leading term of the (usual) monomial $\sigma_2$ in~$\tilde{M}$. In particular, the sign of this monomial is determined exactly by the sign of the permutation~$\sigma_2$ because the entries of $\tilde{M}_{i\sigma_2(i)}$ all have positive leading term, and $\text{val}(B) = 0$.

   The part $Ax$ consists of permutations of $S_n$ that does include $\tilde{M}_{i\sigma_1(i)}$, and in particular this includes the monomial $\sigma_1$.  Since every $\tilde{M}_{j\sigma_1(j)}$ has valuation 0 by construction (and every other monomial has higher valuation by assumption), the leading term of $Ax$ comes from the leading term of $\sigma_1$. As above, the sign of this monomial is the sign of the permutation $\sigma_1$, and $\text{val}(A) = 0$. Thus, $x = -B/A$ has valuation $0-0 = 0$, and it is positive because $B$ and $A$ have differing signs, and so this yields a lift of $M$ to $\tilde{M}$, a singular matrix of positive real Puiseux series, as desired.
\end{proof}

\begin{corollary}
\label{cor:corank1}
    \label{cor:non-symm-hypersurface-real}
    The real tropicalization $\tropRg(\cM_{n,n}^{n-1})$ is equal to the tropicalization $\tropCg\cM_{n,n}^{n-1}$.
\end{corollary}

\begin{proof}
    This follows exactly from the proof of Theorem~\ref{thm:nonsymm-hypersurface-real-positive}, except we are no longer concerned with $\sigma_1$ and $\sigma_2$ having different sign, nor with the signs of $A$, $B$, and $x$.
\end{proof}

More generally, it was shown in Example: Harmony of \cite{VINZANT2012392} that for any $2 \leq d \leq n$, the real and complex tropicalization coincide for the variety of $d \times n$ matrices of rank $< d$.

\section{Symmetric Corank 1}

The variety $\cS^{n-1}_n$ of singular symmetric $n \times n$ matrices defined by the $n\times n$ symmetric determinant, which is a polynomial in ${n+1 \choose 2}$ variables. Each monomial corresponds to permutations in $S_n$ up to forgetting the orientations of cycles. For example, in the $4 \times 4$ symmetric determinant, the permutations $(1\,2\,3\,4)$ and $(4\,3\,2\,1)$ both correspond to the monomial $m_{12}m_{23}m_{34}m_{14}$, but $(1\,3\,2\,4)$ gives a different monomial$m_{13}m_{14}m_{23}m_{24}$. The sign of a monomial is simply the sign of any permutation that could make up that monomial.

We can express the exponent vector of each monomial in the symmetric determinant as an upper triangular matrix $M$ of $0$s, $1$s, and $2$s. This symmetrization map $M\mapsto \frac{1}{2}(M+M^\top)$ is a linear isomorphism from the vector space of upper triangular matrices to the vector space of symmetric matrices, and we can study our Newton polytope in either of these vector spaces.  For a permutation matrix $P$, the symmetrization $\frac{1}{2}(P + P^\top)$ has every row and column sum equal to 1.

By \cite{Brualdi2017SymmetricHA} the vertices of this polytope correspond to permutations  without even cycles of length at least $ 4$, up to reorientation of cycles. There is a bijection between monomial terms of the symmetric determinant and semisimple graphs on $n$ vertices: for a permutation $\sigma$ let $G_\sigma$ on vertices~$[n]$ with an edge between $i$ and $\sigma(i)$ for every $i$. If $i$ is fixed by $\sigma$, then the graph has a loop at~$i$. The connected components of $G_\sigma$ are simply the cycles in the cycle decomposition of $\sigma$.
Then the vertices of the Newton polytope of the symmetric determinant are in bijection with semisimple graphs whose connected components consist only of loops (fixed points of a permutation), isolated edges (transpositions), and cycles of odd length (cycles of odd length without orientation). By Corollary 1.4 of \cite{ahmed2008polytopes}, the edges of the Newton polytope are pairs of vertices $(u,v)$ where the edge sets of their graphs $G_u = ([n], E_u)$ and $G_v = ([n], E_v)$ satisfy $|E_u \cup E_v| \le n+1$, and $G_{u,v} = ([n], E_u \cup E_v)$ contains at most one even cycle of length $\ge 4$.

\begin{example}
   Some examples of monomial terms of symmetric $4 \times 4$ determinant are shown in Table~\ref{tab:symm-det-monomials}. 
    All monomials in the table except the last one are vertices in the Newton polytope. The last monomial is the midpoint of the edge formed by the two transpositions. Every pair of vertices in the table forms an edge except $2x_{12}x_{13}x_{23}x_{44}$ and $-x_{11}x_{22}x_{34}^2$, whose the graph union has 7 edges. The edges between $2x_{12}x_{13}x_{23}x_{44}$ and the monomials that are the product of two disjoint 
transpositions is not in $\tropC (\cS_4^3)$, because the monomials have the same sign. 

The lattice length of an edge is one less than the number of lattice points in the edge.
The edges between $-x_{11}x_{22}x_{34}^2$ and the product-of-transpositions monomials have lattice length 1 and are in $\tropC(\cS_4^3)$. The edge formed by the two product-of-transpositions monomials has lattice length 2, with a 4-cycle as the midpoint. 
\end{example}

    \renewcommand\theadalign{bc}
    \renewcommand\theadfont{\bfseries}
    \renewcommand\theadgape{\Gape[7pt]}
    \renewcommand\cellgape{\Gape[7pt]}
    \setlength{\abovecaptionskip}{300pt}
    \begin{table}[h]
        \caption{Several monomials of the $4 \times 4$ symmetric determinant and equivalent ways to represent them.}
    \label{tab:symm-det-monomials}
        \renewcommand{\arraystretch}{1.2}
        \begin{tabular}{|*5{>{\renewcommand{\arraystretch}{1.2}}c|}}
        \hline
        Monomial & Permutation(s) & Upper Triangular & Symmetric Matrix & Graph\\
        \hline   
        \thead{
        $2x_{12}x_{13}x_{23}x_{44}$} & $(123)(4) \equiv (132)(4)$ & 
        \thead{${\begin{pmatrix}0 & 1 & 1 & 0\\ & 0 & 1 & 0\\ & & 0 & 0\\ & & & 1\end{pmatrix}}$} & $\thead{\begin{pmatrix}0 & \frac{1}{2} & \frac{1}{2} & 0 \\\frac{1}{2} & 0 & \frac{1}{2} & 0\\\frac{1}{2} & \frac{1}{2} & 0 & 0\\0 & 0 & 0 & 1\end{pmatrix}}$ & \begin{tikzpicture}[baseline=(current bounding box.east)]
        \node[draw, circle, inner sep=1pt] (1) {$1$};
        \node[draw, circle, inner sep=1pt] (2) [right of=1] {$2$};
        \node[draw, circle, inner sep=1pt] (3) [below of=1] {$3$};
        \node[draw, circle, inner sep=1pt] (4) [below of=2] {$4$};

        \draw (1) -- (2);
        \draw (1) -- (3);
        \draw (2) -- (3);
        \draw (4) to [out=90, in=180, looseness=5] (4);
        \end{tikzpicture}\\
        \hline
        $-x_{11}x_{22}x_{34}^2$ & $(1)(2)(34)$ & ${\thead{\begin{pmatrix}1 & 0 & 0 & 0\\&1 & 0 & 0\\& & 0 & 2\\& & & 0\end{pmatrix}}}$ & $\thead{\begin{pmatrix}1 & 0 & 0 & 0\\0 & 1 & 0 & 0\\0 & 0 & 0 & 1\\0 & 0 & 1 & 0\end{pmatrix}}$ & \begin{tikzpicture}[baseline=(current bounding box.east)]
        \node[draw, circle, inner sep=1pt] (1) {$1$};
        \node[draw, circle, inner sep=1pt] (2) [right of=1] {$2$};
        \node[draw, circle, inner sep=1pt] (4) [below of=1] {$4$};
        \node[draw, circle, inner sep=1pt] (3) [below of=2] {$3$};

        \draw (1) to [out=0, in=270, looseness=5] (1);
        \draw (2) to [out=180, in=270, looseness=5] (2);
        \draw (3) -- (4);
        \end{tikzpicture}\\
        \hline
        $x_{12}^2x_{34}^2$ & $(12)(34)$ & $\thead{\begin{pmatrix} 0 & 2 & 0 & 0\\& 0 & 0 & 0\\& & 0 & 2\\& & & 0\end{pmatrix}}$ & $\thead{\begin{pmatrix} 0 & 1 & 0 & 0\\1 & 0 & 0 & 0\\0 & 0 & 0 & 1\\0 & 0 & 1 & 0\end{pmatrix}}$& \begin{tikzpicture}[baseline=(current bounding box.east)]
        \node[draw, circle, inner sep=1pt] (1) {$1$};
        \node[draw, circle, inner sep=1pt] (2) [right of=1] {$2$};
        \node[draw, circle, inner sep=1pt] (4) [below of=1] {$4$};
        \node[draw, circle, inner sep=1pt] (3) [below of=2] {$3$};

        \draw (1) -- (2);
        \draw (3) -- (4);
        \end{tikzpicture}\\
        \hline
        $x_{14}^2x_{23}^2$ & $(14)(23)$ & $\thead{\begin{pmatrix} 0 & 0 & 0 & 2\\& 0 & 2 & 0\\& & 0 & 0\\ & & & 0\end{pmatrix}}$ & $\thead{\begin{pmatrix}0 & 0 & 0 & 1\\0 & 0 & 1 & 0\\0 & 1 & 0 & 0\\1 & 0 & 0 & 0\end{pmatrix}}$ &\begin{tikzpicture}[baseline=(current bounding box.east)]
        \node[draw, circle, inner sep=1pt] (1) {$1$};
        \node[draw, circle, inner sep=1pt] (2) [right of=1] {$2$};
        \node[draw, circle, inner sep=1pt] (4) [below of=1] {$4$};
        \node[draw, circle, inner sep=1pt] (3) [below of=2] {$3$};

        \draw (1) -- (4);
        \draw (2) -- (3);
        \end{tikzpicture}\\
        \hline
        $-2x_{12}x_{14}x_{23}x_{34}$ & $(1243) \equiv (3421)$ & $\thead{\begin{pmatrix}0 & 1 & 0 & 1\\& 0 & 1 &0\\& & 0 & 1\\& & &0\end{pmatrix}}$ & $\thead{\begin{pmatrix}0 & \frac12 & 0 & \frac12\\\frac12 & 0 & \frac12 & 0\\0 & \frac12 & 0 & \frac12\\\frac12 & 0 & \frac12 & 0\end{pmatrix}}$&\begin{tikzpicture}[baseline=(current bounding box.east)]
        \node[draw, circle, inner sep=1pt] (1) {$1$};
        \node[draw, circle, inner sep=1pt] (2) [right of=1] {$2$};
        \node[draw, circle, inner sep=1pt] (4) [below of=1] {$4$};
        \node[draw, circle, inner sep=1pt] (3) [below of=2] {$3$};

        \draw (1) -- (2);
        \draw (3) -- (4);
        \draw (1) -- (4);
        \draw (2) -- (3);
        \end{tikzpicture}\\
        \hline
    \end{tabular}
    \renewcommand{\arraystretch}{1}
    \end{table}

\begin{lemma}\label{lem:edges-of-Newt}
    The edges of the Newton polytope of the symmetric determinant have lattice length either 1 or 2. The edges of lattice length 2 are the convex hull of two vertices that differ on exactly $k$ transpositions for $k \ge 2$, and have as midpoint a monomial of the symmetric determinant that agrees with the two vertex monomials everywhere except on those $k$ transpositions, where it has a $2k$-cycle instead.
\end{lemma}

\begin{proof}
    If $u$ and $v$ are vertices and $(u,v)$ is an edge in the Newton polytope of the determinant, denoted by $\mathcal{N}(\det(M))$, then considering $G_u = ([n], E_u)$ and $G_v = ([n], E_v)$, we have that $E_u \cup E_v$ contains at most one even cycle of length $\ge 4$.
    
    We consider whether or not the union $E_u \cup E_v$ contains an even cycle, and we will see that these determine the two types of edges of the polytope. Considering $a_u$ and $a_v$ as the exponent vectors of $u$ and $v$ respectively, $u$ and $v$ form an edge that contains a lattice point on the interior if and only if the gcd of the nonzero entries of the vector $a_u - a_v$ is greater than one.  
    Since the entries of $a_u$ and $a_v$ can only lie in $\{0, 1, 2\}$, the vector $a_u-a_v$ has entries in $\{0,\pm1,\pm2\}$, so the gcd $> 1$ only when $u-v$ contains only entries in $\{0, 2, -2\}$, and if so, the only possible interior lattice point is the midpoint $\frac12(a_u+a_v)$.  The exponent 2 in the monomial correspons to a transposition in the permutation (i.e. a connected component consisting of an isolated edge.) 
    In terms of graphs we have $G_{\frac12(u+v)} = ([n], E_u \cup E_v)$: every edge in the graphs of $G_u$ or $G_v$ is also an edge in $G_{\frac12(u+v)}$. Thus, an  isolated edge that belongs to $G_u$ but not to $G_v$, or vice versa, becomes an edge of a cycle.

    These new cycles cannot be odd: we can bipartition them based on whether they come from $G_u$ or $G_v$, and odd cycles are not 2-edge-colorable. Furthermore, if there are multiple new even cycles of length greater than or equal to $4$, then $u$ and $v$ did not form an edge, so this cycle must be unique and of even length $\ge 4$.
\end{proof}

The positive tropicalization of a variety $X$ defined by the vanishing of an ideal $I \subset \cC[x_1,\dots,x_n]$ has the following characterization
\[
\tropC(X) = \{w \in \RR^n \mid \mathop{in_w}(I) \text{ has a positive real zero}\}
\]
When $I$ is generated by a single polynomial $f= \sum_{a \in A} c_a x^a$, with all $c_a \neq 0$, the initial ideal $\mathop{in_w}(I)$ with respect to weight $w \in \RR^n$ is generated by the initial form 

$\mathop{in_w} f = \sum c_a x^a$ where the sum is taken over all $\{a \in A \mid a \cdot w \leq b\cdot w \, \forall b \in A\}.$

\begin{theorem}\label{thm:tropC+-for-symm-det}
    Let $(u,v)$ be an edge of the Newton polytope of the symmetric determinant, between two permutations $u$ and $v$. Its normal cone lies in $\tropC(\cS_n^{n-1})$ if and only if of the following holds:
    \begin{enumerate}
        \item It has lattice length 1, and $u$ and $v$ have opposite sign.
        \item It has lattice length 2, and the unique even cycle of length $\ge 4$ on its midpoint monomial has length divisible by $4$.
    \end{enumerate}
    In other words, $\tropC(\cS_n^{n-1})$ is the union of normal cones to edges containing a positive and negative monomial, or equivalently a positive and negative permutation.
\end{theorem}

\begin{proof}If the edge has lattice length $1$, then the initial ideal is, for some integer $c$, $\langle cx^u - x^v\rangle$ if $u$ and $v$ have opposite signs and is $\langle cx^u + x^v\rangle$ if $u$ and $v$ have the same sign.  It has a positive real root only in the case of opposite signs.

   If the edge has lattice length $2$, then $u$ and $v$ have the same sign, as their symmetric difference consists of an even number of transpositions. Their midpoint $\frac12(u+v)$ consists of a $2k$ cycle instead of the $k$ transpositions. Trading $k$ transpositions for an even cycle amounts to changing the sign $k+1$ times. When $k$ is odd, the midpoint has the same sign as $u$ and $v$, so the initial ideal cannot have a positive zero.
    When $k$ is even, the midpoint has opposite sign, and the initial ideal has form $\langle c(x^u - 2x^{\frac12(u+v)} + x^v) \rangle = \langle c(x^{\frac{u}2} - x^{\frac{v}2})^2 \rangle$ where $c$ is a monomial, so it has positive real root.
\end{proof}

Now we consider $\tropR(\cS_n^{n-1})$. We will first show in an example that $\tropR(\cS_n^{n-1}) \neq \tropC(\cS_n^{n-1})$ for $n \ge 4$, and then give a description of $\tropR(\cS_n^{n-1})$ in Theorem~\ref{thm:tropR+-for-symm-hypersurface}.

\begin{example}\label{ex:no-real-positive-lift}
    Consider the following matrix
    $$M = \begin{bmatrix} 2 & \red{0} & 1 & \blue{0} \\
    \red{0} & 2 & \blue{0} & 2\\
    1 & \blue{0} & 2 & \red{0}\\
    \blue{0} & 2 & \red{0} & 1\end{bmatrix}.$$

    It has minimum achieved at the even cycle $(1234)$, and therefore also at the pairs of transpositions $(12)(34)$, in red, and $(14)(23)$, in blue. By Theorem~\ref{thm:tropC+-for-symm-det}, it is in $\tropC(S_4^3)$.

    Any lift of $M$ into $\cR^+$ must take the form
    $$\tilde{M} = \begin{bmatrix} c_{11}t^2 & c_{12} & c_{13}t & c_{14}\\
    c_{12} & c_{22}t^2 & c_{23} & c_{24}t^2\\
    c_{13}t & c_{23} & c_{33}t^2 & c_{34} \\
    c_{14} & c_{24}t^2 & c_{34} & c_{44}t\end{bmatrix},$$
    where the $c_{ij}$ are all elements of $\cR^+$ with valuation zero. Further, this lift must make $\tilde{M}$ singular. As in the proof of Theorem~\ref{thm:nonsymm-hypersurface-real-positive}, we replace $c_{12}$ with a variable $x$, and examine the determinant of~$\tilde{M}$ as a quadratic polynomial in $x$. We want a root of this polynomial in $\cR^+$ with valuation zero.  So we want its discriminant to be nonnegative for some choice of $c_{ij}$s.
    
 By explicit expansion of the determinant, and using the fact that the $c_{ij}$'s and $x$ must have valuation $0$, we can check that the discriminant with respect to $x$ of the determinant of $\tilde{M}$ has the leading term $-8c_{13}c_{14}c_{23}^2c_{34}c_{44}t^2$. No choice of (positive) real Puiseux series for the $c_{ij}$s with valuation $0$ can make this term positive, so the discriminant must be negative, and there is no lift of $M$ into $\cR^+$, so $M \not\in \tropR(\cS^3_4).$
\end{example}

The lineality space of the tropical variety is spanned by matrices $\cM^{(k)} \in \RR^{n \times n}$ of the following form, for every $k \in [n]$, $$M_{ij}^{(k)} = \begin{cases} 2 & i=j=k,\\ 1 & i=k \text{ or } j=k, i \neq j, \\ 0 &\text{otherwise.}\end{cases}$$
which represents simultaneous tropical scaling of row $k$ and column $k$.

\begin{theorem}\label{thm:tropR+-for-symm-hypersurface}
    A $n \times n$ symmetric tropical matrix $M$ lies in $\tropR(\cS_n^{n-1})$ if and only if the minimum among the tropical monomials of the symmetric tropical determinant on $M$ satisfies one of the following:
    \begin{enumerate}
        \item the minimum is achieved at two monomials $u$ and $v$ that have opposite sign and form an edge of lattice length 1.
        \item the minimum is achieved at three monomials $u$, $v$, and $\frac12(u+v)$ where $u$ and $v$ differ along $2k$ transpositions for some $k\ge1$, and $\frac12(u+v)$ differs from $u$ and $v$ along a $4k$-cycle, and for any pair of adjacent elements $i,j$ on that cycle, the minimums of the tropical minors of $M_i$ and $M_j$ are achieved on monomials of the same sign, where $M_i$ is $M$ with the $i$-th row and $i$-th column deleted.
    \end{enumerate}
\end{theorem}

In case (2), it is sufficient to check a single adjacent pair $i,j$ on the cycle. 

\begin{proof}
    We first consider the case where the minimum in the monomials of the symmetric tropical determinant of $M$ is achieved on two monomials of opposite sign forming an edge of lattice length~1. 
    The normal vectors to the edges of lattice length 1 gives the initial ideal $\langle cx^u - x^v \rangle$, which, after monomial change of basis, is equivalent to $\langle cx_1-x_2 \rangle$, for some $c > 0$. By \cite[Lemma 2.6]{VINZANT2012392}, since this is everywhere smooth, the cone dual to this edge is in $\tropRg(\cS_n^{n-1})$. Furthermore, any solution to $cx_2-x_1$ is (after possibly multiplying by $-1$, and modifying $x_i, i \ge 3$) a positive solution, so the dual cone is in $\tropR(\cS_n^{n-1})$. 

    Now we consider the edges of lattice length 2. By relabeling, we can assume that $u$ and $v$ agree everywhere except on the $4k$-cycle $(1 \, 2\, 3\, \ldots \,4k)$ (up to cycle orientation). Again assume without loss of generality that $u$ contains the transpositions $(12)(34)\ldots((4k-1)(4k))$, and $v$ contains the transpositions $(23)(45)\ldots((4k)1)$. Via tropical scaling, we can also assume that $M_{12} = M_{21} = 0$.

    Again let us consider an arbitrary (generic) lift of the entries of $M$ to a symmetric matrix of positive real Puiseux series, $\tilde{M}$, leaving indeterminate $x =\tilde{M}_{12} = \tilde{M}_{21}$.  We view the determinant $\det(\tilde{M})$ as a quadratic equation in $x$, and we will describe when $x$ has a positive real Puiseux series lift of valuation 0.

    We have $\det(\tilde{M})(x) = Ax^2 + Bx + C$. Again let $c$ be the minimum valuation of monomials in $\det(\tilde{M})$. As before, the term $Ax^2$ captures all monomials of the symmetric determinant that includes both $M_{12}$ and $M_{21}$. This includes the vertex $u$, and the leading term of $Ax^2$ is the leading term of the monomial corresponding to $u$, so $\text{val}(A) = c$ and $A$ shares sign with $u$. The term $Bx$ captures the midpoint $\frac12(u+v)$, which contains the $4k$-cycle $(1\,2\,3\,\ldots\,4k)$. Thus, $\text{val}(B) = c$. The term $C$ includes $v$, so $\text{val}(C) = c$ and $A$ and $C$ share signs with $u$ and $v$ and each other, and opposite signs with $B$.

    As $\det(\tilde{M})(x)$ is a quadratic equation in $x$, it roots in $\cC$ have the form $$x = \frac{-B \pm \sqrt{B^2 - 4AC}}{2A}.$$
    These solutions are in $\cR$ when the discriminant is nonnegative. Furthermore, if the discriminant is positive, since we see that $A$ and $C$ share sign, $|B| > \sqrt{B^2 - 4AC} > 0$.  Therefore, both $-B\pm \sqrt{B^2-4AC}$ have the same sign as $-B$, and, based on the sign of $-B$, at least one of $\pm\sqrt{B^2-4AC}$ has the same sign as $-B$, and so adding it to $-B$ cannot cause cancellation of the leading term. Thus, without loss of generality, $\text{val}(-B + \sqrt{B^2-4AC}) = \text{val}(-B) = \text{val}(B)$. Thus, there is a choice of $x$ so that $\text{val}(x) = \text{val}(B) - \text{val}(2A) = c - c = 0$, as desired, and as $-B$ and $A$ have the same sign, $x$ is positive.
    
    It remains to show that the discriminant is positive if and only if the minimum of $\trop\det(M_1)$ and $\trop\det(M_2)$ are achieved by monomials of the same sign. From Lemma~\ref{lem:factoring-discriminant}, the discriminant $\text{Disc}_{m_{12}}(\det(\tilde{M}))$ is equal to $\det(\tilde{M}_1)\det(\tilde{M}_2)$. 
    We can choose lifts of rows 1 and 2 of $\tilde{M}$ so that the signs of $\det\tilde{M}_1$ and $\det\tilde{M}_2$ are equal to signs of any choice of permutations attaining the minimum in $\trop\det(M_1)$ and $\trop\det(M_2)$ respectively.
    In particular, we can choose them to have the same sign, so that the discriminant $\text{Disc}_{m_{12}}(\det(\tilde{M}))$ is positive, and $\tilde{M}$ is a lift of $M$ as desired.

    If $\trop\det(M_1)$ and $\trop\det(M_2)$ have their minimums attained only on monomials of opposite signs, then any lifts of $M_1$ and $M_2$ must have $\det(\tilde{M_1})$ and $\det(\tilde{M_2})$ have opposite signs because cancellation cannot occur. Thus, $\text{Disc}_{m_{12}}(\det(\tilde{M}))$ must be negative.
\end{proof}

\begin{theorem}
    \label{thm:SingularSymmetric} For the variety $\cS_n^{n-1}$ of symmetric singular $n \times n$ matrices, the real and complex tropicalizations coincide, that is,
    \[\tropRg(\cS_n^{n-1}) = \tropCg(\cS_n^{n-1}).\]
\end{theorem}

\begin{proof}
    The edges of lattice length 1 have initial ideal $\langle cx^u + x^v \rangle$ for $c \in \RR$, so after monomial change of basis we have $\langle cx_1+x_2 \rangle$. The normal cone to this edge is in $\tropRg(\cS_n^{n-1})$ by \cite[Lemma 2.6]{VINZANT2012392}.
  For edges of lattice length $2$, we follow from the same argument in the proof of Theorem~\ref{thm:tropR+-for-symm-hypersurface}.  However now the sign of $\tilde{M}_1\tilde{M}_2$ will always be positive for appropriate choice of the signs of $\tilde{M}_{ij}$: every monomial in $\det{\tilde{M}_1}$ contains some element $\tilde{M}_{2j}$ from the second column of $\tilde{M}$, and no monomial in $\det{\tilde{M}_2}$ contains any, so we can simply flip the sign of $\tilde{M}_{2j}$ if $\tilde{M}_1\tilde{M}_2$ is negative. For edges of lattice length 2 where the even cycle in $\frac12(u+v)$ has length $4k+2$, the same argument works directly, only the sign of $x$ may change, which doesn't matter for $\tropRg(\cS_n^{n-1})$.
\end{proof}

\section*{Acknowledgements}
We are grateful to Marie-Charlotte Brandenburg, Georg Loho, and Rainer Sinn for discussions about the paper~\cite{BLS} and related topics, and to Felipe Rinc\'on and Cynthia Vinzant for discussions about real tropicalizations. We thank Jonathan Leake for the observation of Lemma~\ref{lem:factoring-discriminant} and Changxin Ding for pointing us to the example in the proof of Lemma~\ref{lem:rnk3}. 
MC and JY were partially supported by NSF grant \#1855726. 
\bibliography{main}

\begin{thebibliography}{Ahm08}

\bibitem[AGS20]{allamigeon2020tropical}
Xavier Allamigeon, St{\'e}phane Gaubert, and Mateusz Skomra.
\newblock Tropical spectrahedra.
\newblock {\em Discrete \& Computational Geometry}, 63:507--548, 2020.

\bibitem[Ahm08]{ahmed2008polytopes}
Maya~Mohsin Ahmed.
\newblock Polytopes of magic labelings of graphs and the faces of the birkhoff
  polytope.
\newblock {\em Annals of Combinatorics}, 12:241--269, 2008.

\bibitem[Ard04]{ardila2004tropical}
Federico Ardila.
\newblock A tropical morphism related to the hyperplane arrangement of the
  complete bipartite graph, 2004.

\bibitem[BA96]{BL96}
Louis Billera and Sarang Aravamuthan.
\newblock The combinatorics of permutation polytopes.
\newblock {\em DIMACS Series in Discrete Mathematics and Theoretical Computer
  Science}, 24, 01 1996.

\bibitem[BC17]{Brualdi2017SymmetricHA}
Richard~A. Brualdi and Lei Cao.
\newblock Symmetric, hankel-symmetric, and centrosymmetric doubly stochastic
  matrices.
\newblock {\em Acta Mathematica Vietnamica}, 43:675--700, 2017.

\bibitem[Ber17]{Bernstein_rank2}
Daniel~Irving Bernstein.
\newblock Completion of tree metrics and rank 2 matrices.
\newblock {\em Linear Algebra Appl.}, 533:1--13, 2017.

\bibitem[BLS23]{BLS}
Marie-Charlotte Brandenburg, Georg Loho, and Rainer Sinn.
\newblock Tropical positivity and determinantal varieties.
\newblock {\em Algebr. Comb.}, 6(4):999--1040, 2023.

\bibitem[BLS24]{BLS_corrigendum}
Marie-Charlotte Brandenburg, Georg Loho, and Rainer Sinn.
\newblock Corrigendum to {{\textquotedblleft}Tropical} {Positivity} and
  {Determinantal} {Varieties{\textquotedblright}}.
\newblock {\em Algebraic Combinatorics}, 7(3):749--751, 2024.

\bibitem[CC12]{CartwrightChan}
Dustin Cartwright and Melody Chan.
\newblock Three notions of tropical rank for symmetric matrices.
\newblock {\em Combinatorica}, 32(1):55--84, 2012.

\bibitem[CLY24]{cai2024symmetric}
May Cai, Kisun Lee, and Josephine Yu.
\newblock Symmetric tropical rank 2 matrices.
\newblock arXiv:2404.08121, 2024.

\bibitem[Dod67]{Dodgson}
Charles~Lutwidge Dodgson.
\newblock Condensation of determinants, being a new and brief method for
  computing their arithmetical values.
\newblock {\em Proc. R. Soc. Lond.}, 15:150--155, 1867.

\bibitem[DSS05]{develin2005rank}
Mike Develin, Francisco Santos, and Bernd Sturmfels.
\newblock On the rank of a tropical matrix.
\newblock {\em Combinatorial and computational geometry}, 52:213--242, 2005.

\bibitem[FR15]{FinkRincon}
Alex Fink and Felipe Rinc\'on.
\newblock Stiefel tropical linear spaces.
\newblock {\em J. Combin. Theory Ser. A}, 135:291--331, 2015.

\bibitem[HJ10]{HultmanJonsson}
Axel Hultman and Jakob Jonsson.
\newblock The topology of the space of matrices of {B}arvinok rank two.
\newblock {\em Beitr\"age Algebra Geom.}, 51(2):373--390, 2010.

\bibitem[JSY20]{JSY}
Philipp Jell, Claus Scheiderer, and Josephine Yu.
\newblock {Real Tropicalization and Analytification of Semialgebraic Sets}.
\newblock {\em International Mathematics Research Notices}, 2022(2):928--958,
  05 2020.

\bibitem[MS15]{MaclaganSturmfels}
Diane Maclagan and Bernd Sturmfels.
\newblock {\em Introduction to tropical geometry}, volume 161 of {\em Graduate
  Studies in Mathematics}.
\newblock American Mathematical Society, Providence, RI, 2015.

\bibitem[MY09]{MarkwigYu}
Hannah Markwig and Josephine Yu.
\newblock The space of tropically collinear points is shellable.
\newblock {\em Collect. Math.}, 60(1):63--77, 2009.

\bibitem[Oxl11]{Oxley}
James Oxley.
\newblock {\em Matroid theory}, volume~21 of {\em Oxford Graduate Texts in
  Mathematics}.
\newblock Oxford University Press, Oxford, second edition, 2011.

\bibitem[PS04]{PachterSturmfels}
Lior Pachter and Bernd Sturmfels.
\newblock Tropical geometry of statistical models.
\newblock {\em Proc. Natl. Acad. Sci. USA}, 101(46):16132--16137, 2004.

\bibitem[RT24]{rose2024computing}
Kemal Rose and M{\'a}t{\'e}~L Telek.
\newblock Computing positive tropical varieties and lower bounds on the number
  of positive roots.
\newblock {\em arXiv preprint arXiv:2408.15719}, 2024.

\bibitem[SW05]{SpeyerWilliams}
David Speyer and Lauren Williams.
\newblock The tropical totally positive {G}rassmannian.
\newblock {\em J. Algebraic Combin.}, 22(2):189--210, 2005.

\bibitem[Tab15]{tabera2015real}
Luis~Felipe Tabera.
\newblock On real tropical bases and real tropical discriminants.
\newblock {\em Collectanea mathematica}, 66:77--92, 2015.

\bibitem[Vin12]{VINZANT2012392}
Cynthia Vinzant.
\newblock Real radical initial ideals.
\newblock {\em Journal of Algebra}, 352(1):392--407, 2012.

\bibitem[Zwi21]{zwick2021symmetric}
Dylan Zwick.
\newblock Symmetric kapranov and symmetric tropical ranks, 2021.

\end{thebibliography}
\bibliographystyle{alpha}
\vspace*{1cm}
\end{document}